\documentclass[12pt,twoside,a4paper,reqno]{amsart}
\usepackage{amssymb}
\usepackage{amsmath,amscd}
\usepackage{bbm}

\usepackage[UKenglish]{babel}
\usepackage{inputenc} 

\usepackage{setspace}
\usepackage{hyperref}
\usepackage{aliascnt}
\usepackage{enumitem}
\usepackage{xcolor}

\usepackage{geometry}
\usepackage{mathrsfs}
\usepackage{calligra}
\usepackage[arrow, matrix, curve]{xy}
\usepackage{graphicx}

\makeatletter
	\def\section{\@startsection{section}{1}
		\z@{3\linespacing\@plus.5\linespacing}{\linespacing\@plus .5\linespacing}
		{\normalfont\bf\centering}}

	\newcommand{\myhypertarget}[1]{\Hy@raisedlink{\hypertarget{#1}{}}}

	\newcommand{\CustomLabel}[2]{%
		\protected@write \@auxout {}{\string \newlabel {#1}{{#2}{\thepage}{#2}{#1}{}} }%
		\myhypertarget{#1}{#2}%
		}

	\def\getTitle{\@title}

\makeatother

\mathcode`l="8000
\begingroup
\makeatletter
\lccode`\~=`\l
\DeclareMathSymbol{\lsb@l}{\mathalpha}{letters}{`l}
\lowercase{\gdef~{\ifnum\the\mathgroup=\m@ne \ell \else \lsb@l \fi}}%
\makeatother
\endgroup

\allowdisplaybreaks

\addto\extrasUKenglish{}
\addto\extrasUKenglish{} 

\numberwithin{equation}{section}
\numberwithin{figure}{section}

\newaliascnt{thmX}{equation}

\newtheorem{thm}[thmX]{Theorem}

\newaliascnt{lmaX}{equation}

\newtheorem{lem}[lmaX]{Lemma}

\newaliascnt{corX}{equation}

\newtheorem{cor}[corX]{Corollary}

\newaliascnt{propX}{equation}

\newtheorem{prop}[propX]{Proposition}

\theoremstyle{definition}

\newaliascnt{defX}{equation}

\newtheorem{df}[defX]{Definition}

\newaliascnt{remX}{equation}

\newtheorem{rem}[remX]{Remark}

\newaliascnt{questX}{equation}

\newcommand{\Thm}{Thm~}
\newcommand{\Lem}{Lem.\,}
\newcommand{\Prop}{Prop.\,}
\newcommand{\Sec}{Sect.\,}

\newcommand{\Chap}{Chap.\,}

\def\ItemName{CoverItems}
\newcounter{\ItemName}	
\newcommand{\AddItem}[1][\fpeval{\arabic{\ItemName}}]{\stepcounter{\ItemName}\CustomLabel{\ItemName:#1}{(\alph{\ItemName})}}
\newcommand{\GetItem}[1]{\ref{\ItemName:#1}}

\hypersetup{
	pdftoolbar=true,
	pdfmenubar=true,
	pdftitle={\getTitle}, %
	pdfauthor={Martin Sera},
	pdfsubject={Preprint},
	pdfkeywords={},
	pdfborder= 0 0 .01,
	bookmarksnumbered=false,
	breaklinks=true,   
	colorlinks=true,
	linkcolor={[rgb]{0,.2,.4}},
	citecolor={[rgb]{0,.4,0}},
	menucolor={[rgb]{.4,0,0}},
	urlcolor={[rgb]{.4,0,0}},
}

           \DeclareMathOperator{\ReX}{Re}         \DeclareMathOperator{\supp}{supp}  \DeclareMathOperator{\tr}{tr}

\DeclareMathOperator{\sKer}{\mathscr{K}\text{\kern -4pt {\calligra\large er}}\,} \DeclareMathOperator{\sIm}{\mathscr{I}\text{\kern -5pt {\calligra\large m}}\,} \DeclareMathOperator{\sHom}{\mathscr{H}\text{\kern -4pt {\calligra\large om}}\,} \DeclareMathOperator{\sExt}{\mathscr{E}\text{\kern -3pt {\calligra xt}}\,} \DeclareMathOperator{\sCoker}{\mathscr{C}\text{\kern -3pt {\calligra oker}}\,}

\makeatletter
\def\makephX#1#2{\setbox\z@\hbox{#1}\setbox\tw@\null
  \wd\tw@#2\wd\z@ \box\tw@} 
\makeatother
\newcommand{\DiagrammPunkt}{\;.\makephX{\;.}{-1}}

\newcommand{\KE}[1][0]{\mathscr{K}_{E\ifx#10\else{,#1}\fi}}

\def\tilde{\widetilde}

 \newcommand{\aU}{\mathfrak{U}} \newcommand{\aV}{\mathfrak{V}} 

\newcommand{\sC}{\mathscr{C}}         \newcommand{\sK}{\mathscr{K}}     \newcommand{\sS}{\mathscr{S}}

  \newcommand{\CC}{\mathbb{C}}      \newcommand{\RR}{\mathbb{R}} 

\newcommand{\ii}{\mathrm{i}} 
     \newcommand{\dom}[1][]{\mathrm{dom}_{#1}\,}   
 \newcommand{\loc}{\mathrm{loc}}   \newcommand{\reg}{\mathrm{reg}}   \newcommand{\sing}{\mathrm{sing}}  \newcommand{\SPSH}{\mathrm{SPSH}} 

\newcommand{\cf}{cf.\ }
\newcommand{\eg}{e.\,g.\ }  
\newcommand{\ie}{i.\,e.\ } 
\newcommand{\st}{s.\,t.\ }
\newcommand{\ip}{i.\,p.\ }

\newcommand{\nbd}{\nobreakdash-}
\title[About $L^2$- \& $L^{2,\loc}$-cohomologies of spaces with pseudoconvex boundary]{
About top-degree \texorpdfstring{$L^2$- and $L^{2,\loc}$-}{L{\000\136}2- and L\000\136\{2,loc\}-}Dolbeault cohomologies of complex spaces with pseudoconvex boundary
}
\author[M. Sera]{Martin L. Sera}
\address{M. Sera, Faculty of Engineering, Kyoto University of Advanced Science, Kyoto 615-8577, Japan}
\email{sera.martin@kuas.ac.jp}
\dedicatory{In memory of Nikolay Shcherbina}
\date{\today}
\newcommand{\FormelQed}[1][]{~\hfill\raisebox{#1\baselineskip}[0mm][0mm]{\qedhere}}
\newcommand{\K}[1][...]{{\normalfont[[\textcolor{darkred}{\small{#1}}]]}\ClassWarningNoLine{0}{Achtung: \K[#1] gefunden}}
\definecolor{darkred}{rgb}{.77, .01,.2}
\begin{document}
\begin{abstract}
Let $X$ be a complex space of pure-dimension $n$.
For a pseudoconvex relatively compact domain in $X$ with $\sC^3$-smooth boundary and
embedded in a domain of the complex number space,
we prove that the
$L^2$- and $L^{2,\loc}$-Dolbeault $(n,q)$-cohomology groups are vanishing for $q>0$.
Thereby, we include the case that the forms have values in a Nakano semi-positive holomorphic vector bundle.
Using this local vanishing theorem,
we also prove the equivalence
of the $L^2$- and $L^{2,\loc}$-Dolbeault $(n,q)$-cohomology groups of relatively compact domains $\Omega=\{\rho<0\}$ in $X$
which are defined by a $\sC^3$\nbd smooth function $\rho$
which is strictly plurisubharmonic on a neighbourhood of $\partial\Omega$ except of finitely many points.
\end{abstract}
\maketitle
\section{Introduction}
In the area of complex geometry, the Dolbeault cohomology has turned out to be a crucial tool
for studying complex manifolds and vector bundles.
Whereas the theory is well-developed in the smooth setting, many problems are open
when studying Dolbeault cohomology on complex spaces.
In this context, it has been especially fruitful to consider the $L^2$-theory of the Dolbeault operator.
\par
Therefore, the purpose of this work is to study
Dolbeault cohomology groups
with respect to (locally) square-integrable forms (with values in vector bundles)
on domains with pseudoconvex boundary in complex spaces.
Thereby, the key point is that these  domains
are not necessarily strictly pseudoconvex in (all) boundary points.
Before presenting the results, let us first explain the used notations and setting.
\medskip \par
A Hermitian (complex) space $(X,\omega)$ is a reduced complex space $X$
with a Hermitian metric $\omega$ on the regular part $X_\reg$
such that the following holds:
If $x\in X_\sing$ is an arbitrary point, there exist a neighbourhood $U=U(x)\subset X$,
a holomorphic embedding of $U$ in a domain $B$ in $\CC^N$ and
an ordinary smooth Hermitian metric on $B$ whose restriction to $U_\reg$ is equal to $\omega|_{U_\reg}$ (or equivalent to).
Please note we are going to use $\omega$ to denote both, the Hermitian metric and the associated positive $(1,1)$-form on $X_\reg$.
\smallskip \par
The space of square-integrable forms on $X$ is denoted by %
$L _{p,q}^2(X):= L_{p,q}^2 (X_\reg,\omega)$ and
the space of locally square-integrable forms \emph{on $X$} by
$$L_{p,q}^{2,\loc}(X):=\left\{u\in L_{p,q}^{2,\loc} (X_\reg,\omega)\colon u|_U\in L_{p,q}^2 (U)\ \forall\,U{\Subset} X\right\}\subsetneq L_{p,q}^{2,\loc} (X_\reg,\omega).$$
Let $\overline{\partial}_w\colon L^{2}_{p,q}(X) \rightarrow L^{2}_{p,q+1}(X)$  and $\overline{\partial}_{w,\loc}\colon L^{2,\loc}_{p,q}(X) \rightarrow L^{2,\loc }_{p,q+1}(X)$ denote the maximal extensions of the Dolbeault operator, \ie $\overline{\partial}_w$ and $\overline{\partial}_{w,\loc}$ are defined in the sense of distributions.
Although $\overline{\partial}_w$ and $\overline{\partial}_{w,\loc}$ are obviously not defined for all (locally) square-integrable forms,
usually, we do not distinguish in the notation (explicitly)
the domains of definition, $\dom\overline{\partial}_w$ and $\dom\overline{\partial}_{w,\loc}$,
from $L^{2}_{p,q}(X)$  and $L^{2,\loc}_{p,q}(X)$, respectively.
Also, if the context is clear, we denote $\overline{\partial}_w$ and $\overline{\partial}_{w,\loc}$ just by $\overline{\partial}$.
We use the following notation for their Dolbeault cohomology groups:
For $q>0$, we define
\begin{equation*}\begin{split}
H_w^{p,q}(X) &:=\ker \big(\overline{\partial}_w\colon L^{2}_{p,q}(X) \!\rightarrow\! L^{2}_{p,q+1}(X)\big)
\big/\, \overline{\partial}_w (L^{2}_{p,q-1}(X))\quad\hbox{and}\\
H_{w,\loc}^{p,q}(X) &:=
\ker \big(\overline{\partial}_{w,\loc}\colon L^{2,\loc}_{p,q}(X) \!\rightarrow\! L^{2,\loc}_{p,q+1}(X)\big)
\big/\, \overline{\partial}_{w,\loc} (L^{2,\loc}_{p,q-1}(X)).
\end{split}\end{equation*}
\par
Let $X'\subset X_\reg$ be the complement of an analytic set in $X$ containing all singular points $X_\sing$ and
let $E\rightarrow X'$ be a holomorphic vector bundle (just) defined on $X'$ with the Hermitian metric $h$.
An important example for $E$ to be not defined on the whole $X$ (as vector bundle) is
if $E$ is the vector bundle of a coherent analytic sheaf $\sS$ and %
$X'$ are all points in $X_\reg$ where $\sS$ is locally free.
Actually, this is the case if and only if
$E$ is the direct image of a holomorphic vector bundle $\tilde E$ on a proper modification $\pi\colon Y\rightarrow X$
(whereby $\pi_\ast \tilde E$ is not necessarily a vector bundle on all of $X$). %
For forms with values in $E$, %
we can define
$L _{p,q}^2(X, E):=L _{p,q}^2(X',\omega; E,h)$, %
$$L_{p,q}^{2,\loc}(X,E):=\left\{u\in L_{p,q}^{2,\loc} (X',\omega;E,h)\colon u|_U\in L_{p,q}^2 (U, E)\ \forall\,U{\Subset} X\right\},$$ %
$H_w^{p,q}(X,E)$ and $H_{w,\loc}^{p,q}(X,E)$
analogously as above.
Of course, these definitions depend on the metric $h$ on $E$ in general.
Yet, we suppress $h$ in the notation whenever it is clear by context.
\smallskip \par
We say that a domain $\Omega$ in $X$ has a $\sC^k$-smooth boundary
if there exists a $\sC^k$-smooth%
\footnote{
A function $\rho$ on a complex space $X$ is called $\sC^k$-smooth if locally,
there exist an embedding of $X$ in some complex number space and a $\sC^k$-smooth extension of $\rho$ to the ambient space.
}  $\Omega$-defining function $\rho$ on $X$, %
\ie $\Omega=\{\rho<0\}$,
with $d\rho\neq 0$ in a neighbourhood of $\partial \Omega$. %
In particular, $|d\rho|^{-1}$ is uniformly bounded on $\partial \Omega\cap X_\reg$.
If $k\geq 2$ and if $x\in\partial\Omega$ is a pseudoconvex boundary point of $\Omega$
(\ie $U\cap\Omega$ is pseudoconvex%
\footnote{
A domain $D\subset X$ is called pseudoconvex
if there exists a continuous/smooth strictly plurisubharmonic smooth exhaustion function%
{\textsuperscript{\ref{fn:psh}}}, \ie $D$ is a Stein space. %
} for any sufficient small pseudoconvex neighbourhood $U\subset X$),
then we get that the $\ii\partial\overline{\partial}\rho$ is semi-positive definite
on the holomorphic tangent space of $\partial\Omega$ in $x\in X_\reg$.
If $x\in\partial\Omega$ is a strictly pseudoconvex boundary point of $\Omega$, then we can choose $\rho$ to be strictly plurisubharmonic%
\footnote{\label{fn:psh}%
An (upper semi-continuous) function $\varphi$ on a complex space $X$ is plurisubharmonic
if locally, there exist some embedding of $X$ in the complex number space and
a plurisubharmonic extension of $\varphi$ to the ambient space.
Please keep in mind that a smooth plurisubharmonic function does not need to have
a (local) extension which is both smooth and plurisubharmonic.
If any small smooth perturbations of $\varphi$ is also plurisubharmonic, then $\varphi$ is strictly plurisubharmonic.
} on a neighbourhood of $x$.
\medskip \par
The first main theorem of the present work is the following local vanishing result.
\begin{thm}\label{main-local-thm}
Let $X$ be an analytic set of pure-dimension $n$ in the unit ball $B\subset\CC^N$,
let $X'\subset X_\reg$ be the complement of an analytic set in $X$ containing $X_\sing$,
let $E$ be a Nakano semi-positive holomorphic vector bundle on $X'$ and
let $D\Subset X$ be a relatively compact pseudoconvex domain in $X$ with $\sC^3$-smooth boundary.
Then, for all $q>0$,
\[H^{n,q}_{w}(D,E) =0=H^{n,q}_{w,\loc}(D,E).\] %
\end{thm}
Here, the cohomologies are taken with respect to the standard Hermitian metric on $\CC^N$ restricted to $D$.
However, the cohomology groups vanish also for any other chosen metric on $X$
as these are equivalent to the standard metric on $D$.
\par
By generalizing the arguments of W. L. Pardon and M. A. Stern for \cite[\Prop 2.1]{PardonStern91}
to our more general setting,
we prove \autoref{main-local-thm} in \autoref{sec:local}.
\smallskip \par
If $X$ has only isolated singularities which are not contained in $\partial D$,
then $H^{p,q}_{w}(D) =0$ for $p+q>n$ was proven by J.\,E. Forn\ae ss, N. \O vredlid and S. Vassiliadou,
see \cite[\Thm 1.4]{Fornaess-Ovrelid-Vassiliadou05} (together with many other interesting results in this context).
This is a generalization of a result by T. Ohsawa, \cite[\Prop 4.1]{Ohsawa87Hodge}.
It would be interesting to extend their result to the case when singular points of $X$ are contained in $\partial D$.
This might be possible by combining classical techniques of L. H\"ormander
with the ones from T. Ohsawa or Forn\ae ss--\O vredlid--Vassiliadou.
\medskip \par
Using \autoref{main-local-thm},
we can prove the equivalence of $L^2$- and $L^{2,\loc}$-Dolbeault cohomologies, see \autoref{main-grauert} below.
Since we do not need that $E$ is Nakano semi-positive globally for this,
we use the following much weaker notion instead.
\begin{df}
\label{df:loc-Nakano}
Let $X$ be a complex space,
let $X'$ be the complement of an analytic set in $X$ and %
let $(E,h)$ be a Hermitian holomorphic vector bundle on $X'$.
Then, $h$ (and $E$) are called \emph{locally Nakano semi-positive with respect to $X$}
if for every small enough open $U\subset X$, there exists a Nakano semi-positive metric $h'$ on $E|_{U\cap X'}$
which is equivalent to $h|_{U\cap X'}$,
\ie there is a constant $C$
with $\frac1C\|\cdot\|_{h'}\leq\|\cdot\|_h\leq C\|\cdot\|_{h'}$ on $U\cap X'$.
\end{df}
Since $h$ is not defined in $X{\setminus} X'$,
the equivalence of the metrics $h$ and $h'$ is needed to get $L^{2}_{p,q}(U;E,h)=L^{2}_{p,q}(U;E,h')$ and
$L^{2,\loc}_{p,q}(U;E,h)=L^{2,\loc}_{p,q}(U;E,h')$.
So, we can replace $h$ by the Nakano semi-positive $h'$
without changing the space of (locally) square-integrable forms with values in $E$ and corresponding cohomologies.
\par
If $(E,h)$ is defined on the whole $X$, it is trivially satisfied
that $h$ is locally Nakano semi-positive with respect to $X$.
In the more general case that $E=\pi_\ast\tilde E$ on $X'$
for a proper modification $\pi\colon Y\rightarrow X$ and a holomorphic vector bundle $(\tilde E,\tilde h)$ on $Y$,
where $X'$ is the complement of $\pi$'s centre,
we get $h=(\pi^{-1})^\ast({\tilde h})$ on $X'$ is locally Nakano semi-positive with respect to $X$
if for every small enough $U\subset X$, there is a Nakano semi-positive metric $h'$ on $\tilde E|_{\pi^{-1}(U)}$.
Please note that there is no condition on $\tilde h$ itself necessary.
It is just used to ensure that $\pi^\ast h$ has a (smooth) extension to the whole $Y$.
\smallskip \par
Next, let us consider the
sheaf $\KE$ of locally square-integrable \emph{weakly} holomorphic $n$\nbd forms with values in $E$,
\ie $\KE(U)$ is the kernel of $\overline{\partial}_{w,\loc}\colon L^{2,\loc}_{n,0}(U,E)\rightarrow L^{2,\loc}_{n,1}(U,E)$ for any open $U\subset X$.
For the mentioned cohomology equivalence result, we are going to assume that $\KE$ is coherent.
However, this is always the case
if there are a proper modification $\pi\colon Y\rightarrow X$ of $X$ and
a Hermitian holomorphic vector bundle $(\tilde E,\tilde h)$ on $Y$
such that $Y$ is smooth,
$\pi^\ast E=\tilde E|_{\pi^{-1}(X')}$ %
and $\pi^\ast h$ is equivalent to $\tilde h|_{\pi^{-1}(X')}$ (locally with respect to $Y$).
Then, we get $\KE=\pi_\ast(\Omega^n_Y(\tilde E))$
whereby $\Omega^n_Y(\tilde E)$ denotes the holomorphic $n$\nbd forms on $Y$ with values in $\tilde E$,
\cf \cite[\Thm 2.1]{Ruppenthal14Duke} and \cite[\Thm 10.1']{Sera15};
\ip $\KE$ is coherent by Grauert's direct image theorem.
\smallskip \par
\begin{thm}\label{main-grauert}
Let $X$ be a Hermitian space of pure-dimension $n$ and
let $\Omega\Subset X$ be a relatively compact domain in $X$ with $\sC^3$-smooth boundary and
a defining function which is strictly plurisubharmonic on a neighbourhood of $\partial\Omega$ in $X$ besides finitely many points.
Let $X'\subset X_\reg$ be the complement of an analytic set in $X$ and %
let $E$ be a holomorphic vector bundle on $X'$
which is locally Nakano semi-positive with respect to $X$ and
such that $\sK_E$ is a coherent analytic sheaf on $X$.
\par
Then, for all $q>0$,
the natural embedding
$L^2_{n,q}(\Omega,E)\hookrightarrow L^{2,\loc}_{n,q}(\Omega,E)$
induces an isomorphism of the cohomology groups, \ie
\[H^{n,q}_{w}(\Omega,E) \cong H^{n,q}_{w,\loc}(\Omega,E).\]%
\end{thm}
As explained in \autoref{rmk:surjectivity-in-weaker-setting} below,
the surjectivity of $H^{n,q}_{w}(\Omega,E) \rightarrow H^{n,q}_{w,\loc}(\Omega,E)$  %
is correct also in the weaker case that %
$\Omega$'s boundary is strictly pseudoconvex in all points
except of finitely many where it is pseudoconvex at least.
Also, the coherency of $\sK_E$ is not needed for the surjectivity of
$H^{n,q}_{w}(\Omega,E) \rightarrow H^{n,q}_{w,\loc}(\Omega,E)$, see \autoref{ssec:surjectivity}. %
\par
If $X$ is smooth and $\Omega$'s boundary is strictly pseudoconvex,
the equivalence $H^{n,q}_{w}(\Omega) \cong H^{n,q}_{w,\loc}(\Omega)$
is already well known for arbitrary $p$ and $q>0$, %
\cf \eg \Thm 4.1 in \cite[\Chap VIII]{LiebMichel02}.
By adapting the arguments from the smooth case,
the proof of \autoref{main-grauert} is done in \autoref{sec:equivalence}.
\bigskip \par
{\bf Acknowledgments.}{
The author is very grateful to
Jean Ruppenthal and Takeo Ohsawa for fruitful discussions.
This research was supported by the Deutsche Forschungsgemeinschaft (DFG, German Research Foundation),
grant RU 1474/2 within DFG's Emmy Noether Programme and grant SE 2677/1,
and
by JSPS KAKENHI grant number 25K00209.%
}
\medskip \par
\section{Proof of the local vanishing theorem}
\label{sec:local}
\smallskip \par
This section is dedicated to prove \autoref{main-local-thm}.
As mentioned, (many of) the presented arguments are a generalization
of the ones used to prove \Prop 2.1 in \cite{PardonStern91}.
A key ingredient is the following theorem of H. Donnelly and C. Fefferman (see \Thm 1.1 and \Prop 2.1 in \cite{DonnellyFefferman83})
with a simplified proof by T. Ohsawa in \cite[\Thm 1.1]{Ohsawa87Hodge}. %
\begin{thm}[DFO]%
\label{thm:DFO}
Let $(M,\omega)$ be a complete K\"ahler manifold of dimension $n$ and
let $E\rightarrow M$ be a Nakano semi-positive holomorphic vector bundle on $M$.
We assume %
$\omega=\ii \partial\overline{\partial} G$ for a smooth $G\colon M\rightarrow \RR$
with so-called self-bounded gradient, \ie there is a constant $C$ with  $|\partial G|_\omega \leq C$.
If $u$ is a $\overline{\partial}$-closed square-integrable $(n,q)$-form on $M$ with values in $E$,
then there exists a $v\in L^2_{n,q-1}(M,\omega;E)$
with $\overline{\partial} v= u$ and $\|v\|_\omega \leq 4 C \|u\|_\omega$,
\ie $H_w^{n,q}(M,E)=0$ for all $q>0$. %
\end{thm}
J. Ruppenthal proved this for $E$ being a semi-positive line bundle, %
see \cite[\Thm 3.2]{Ruppenthal14Duke}.
His proof generalizes to Nakano semi-positive vector bundles straight forwardly.
\smallskip \par
To apply the DFO theorem, we need a complete metric for the pseudoconvex setting
which is given by the following proposition.
\begin{prop}\label{lem:Ohsawa}
Let $X$ be an analytic set of pure-dimension $n$ in the unit ball $B\subset\CC^N$ and
let $D\Subset X$ be a relatively compact pseudoconvex domain in $X$ with $\sC^3$-smooth boundary.
Then, there exist a $\sC^3$-smooth defining function $\rho$ of $D$ and a positive $\delta >0$
such that $\varphi:=|z|^2-1+\frac \delta{\log (-\rho)}$ is a negative strictly plurisubharmonic function on $D$ and
the Hermitian metric $\ii\partial\overline{\partial}\varphi$ on $D_\reg=D\cap X_\reg$ is complete
with respect to $\partial D$ (\ie closed geodesic balls are relatively compact in $D$).
\end{prop}
This result is stated by T. Ohsawa in \cite[\Lem 2.1]{Ohsawa10} (published in \cite[p.~246]{Ohsawa12})
for pseudoconvex  domains in complex manifolds.
The proof generalizes straight forwardly to complex spaces,
yet we need to assume that the boundary is $\sC^3$-smooth.
Since the mentioned references do not include a proof, let us present the arguments here adjusted to our setting.
\begin{proof}
Let $\rho$ be a $\sC^3$-smooth defining function of $D$, \ie $D=\{\rho<0\}$.
We can assume that $\rho$ is $\sC^3$-smooth on a neighbourhood of $D$ in $\CC^N$.
By calculating the derivatives of $\sigma:=\frac 1{\log(-\rho)}$,
we get
\[
\alpha:=\ii\partial\overline{\partial} \sigma
= \tfrac{\sigma^2}{-\rho}\cdot
\big(\ii\partial\overline{\partial}\rho+\tfrac{1+2\sigma}{-\rho}\ii\partial\rho\wedge\overline{\partial}\rho\big).
\]
To prove the strict plurisubharmonicity of $\varphi$,
we are going to estimate $\alpha$ by $-C\omega_0$ from below
where $C$ is a large constant and $\omega_0:=\ii\partial\overline{\partial}|z|^2$.
\smallskip \par
We can assume %
$|\rho|<e^{-4}$ on $D$, \ie $\sigma>-\frac14$,
such that we obtain
\[\beta:=\tfrac{-\rho}{\sigma^2} \alpha
\geq \big(\ii\partial\overline{\partial}\rho+\tfrac{1}{-2\rho}\ii\partial\rho\wedge\overline{\partial}\rho\big).\]
Since $D$ is pseudoconvex,
$\ii\partial\overline{\partial}\rho$ is non-negative in $\partial D\cap X_\reg$ on the holomorphic tangent space of $\partial D \cap X_\reg$.
$\sC^3$-smoothness of $\rho$ (even in $X_\sing\cap\partial D$) implies that there is a neighbourhood $U$ of $\partial D$ and a constant $C$ such that
\begin{equation}\label{eq:beta-1}
\ii\partial\overline{\partial} \rho\big(x;T,\overline T\big)\geq -C |T|^2
\end{equation}
for all $x\in U^\ast:=U\cap D_\reg$ and all holomorphic tangent vectors $T$ of level sets of $\rho$ in $x$. %
Furthermore, $C$ can be chosen so large such that
\[\ReX \ii\partial\overline{\partial} \rho\big(x;T,\overline N\big)\geq -C|T|\cdot|N|\]
for all $x\in U^\ast$ and for all holomorphic tangent vectors $T$ and normal vectors $N$ (of $\rho$'s level sets) in $x$.
With the ``small constant/large constant'' technique, we get
\begin{equation}\label{eq:beta-2}
2\ReX \ii\partial\overline{\partial} \rho\big(x;T,\overline N\big)
\geq -2C|T|\,|N|\geq 2\rho C|T|^2-\tfrac C{-2\rho}|N|^2.
\end{equation}
By modifying $\rho$, we can assume that
$\ii\partial\rho\wedge \overline{\partial}\rho\big(x;N,\overline N\big)\geq 2C|N|^2$
and the estimates from above are still correct. %
For instance, replace $\rho$ by $\frac1{\tilde C}(e^{\tilde C\rho}-1)$ which is sufficient since $|d\rho|$ is uniformly strictly greater than 0 on $U^\ast$ by assumption (shrink $U$ if necessary).
Hence, we get
\begin{equation}\label{eq:beta-3}
\beta\big(x;N,\overline N\big)\geq \tfrac C{-2\rho} |N|^2 %
\end{equation}
on $U^\ast$.
Let $\xi=T+N$ be a holomorphic tangent vector of $U^\ast$ in $x\in U^\ast$
decomposed in $T$, tangential to the level set of $\rho$ and $N$, normal to the level set.
Combining \eqref{eq:beta-1}, \eqref{eq:beta-2} and \eqref{eq:beta-3}, we obtain %
\[\beta\big(x;\xi,\overline \xi\big)
= \ii\partial\overline{\partial}\rho\big(x;T,\overline T\big)
+  2\ReX\ii\partial\overline{\partial}\rho\big(x;T,\overline N\big) +\beta\big(x;N,\overline N\big)
\geq 3C\rho|T|^2.\] %
Furthermore,
\[
\alpha\big(x;\xi,\overline\xi\big)
=\tfrac {1}{-\rho\log^2(-\rho)}\beta\big(x;\xi,\overline\xi\big)
\geq \tfrac{-3C|T|^2}{\log^2(-\rho)}
\geq -C|T|^2\geq
-C \omega_0\big(x;\xi,\overline\xi\big)
\]
on $U^\ast$. %
Since $\omega=\ii\partial\overline{\partial}\varphi=\ii\partial\overline{\partial}\big(|z|^2+\delta\sigma\big)=\omega_0+\delta\alpha$,
we proved $\varphi$ is strictly plurisubharmonic on $U^\ast$ for $\delta<\frac1C$.
In particular, for any sufficiently small $\delta$, $\varphi$ is strictly plurisubharmonic on the whole $D$.%
\smallskip \par
Next, we show the completeness of $\omega$ with respect to $\partial D$. %
Let $\gamma\colon [0,1]\rightarrow X$ be a differentiable path with $\gamma([0,1))\subset D\cap X_\reg$ and $\gamma(1)\in\partial D$.
Since $\omega=\omega_0+\delta\alpha\geq \frac{\ii\delta\partial\rho\wedge\overline{\partial}\rho}{\rho^2\log^2(-\rho)}$, we get
\[	\|\gamma'(t)\|_\omega
\geq \frac{\delta\|\gamma'(t)\|_{\ii\partial\rho\wedge\overline{\partial}\rho}}{\rho(\gamma(t))\log(-\rho(\gamma(t)))}
= \frac{\delta|(\rho\circ\gamma)'(t)|}{\rho(\gamma(t))\log(-\rho(\gamma(t)))}\]
for $t\in[0,1)$. Therefore, %
\[	\int_0^{1-\varepsilon}\!\!\|\gamma'\|_{\omega} dt
\geq \int_0^{1-\varepsilon}\!\!\!\!\tfrac{\delta|(\rho\circ\gamma)'|}{\rho\circ\gamma\cdot\log(-\rho\circ\gamma)}dt
=\delta\big[\log(-\log(-\rho\circ\gamma))\big]_0^{1-\varepsilon}\rightarrow \infty
\quad\hbox{as }\varepsilon\rightarrow 0,\]
\ie $\gamma$ has infinity length with respect to $\omega$.
\end{proof}
\smallskip \par
The metric $\ii\partial\overline{\partial}\varphi$ from \autoref{lem:Ohsawa} does not necessarily have a self-bounded gradient
as it is required to apply \autoref{thm:DFO} (DFO).
Yet, this can be taken care of with the help of the following elementary lemmata.
\begin{lem}\label{lem:logarithm-self-bounded}
If $\varphi\colon M\rightarrow(-\infty, 0)$ is a $\sC^2$-smooth plurisubharmonic function a complex manifold $M$,
then the plurisubharmonic $\psi:=-\varepsilon\log(-\varphi)$, $\varepsilon>0$, satisfies the inequality
\[{\ii}\partial\psi\wedge \overline{\partial}\psi\leq\varepsilon\ii\partial\overline{\partial}\psi.\] %
\end{lem}
\begin{proof}
This follows directly from calculating the derivatives of $\psi$:
\[\overline{\partial} \psi = \tfrac \varepsilon{-\varphi}\overline{\partial}\varphi \hbox{\ and\ }
\partial\overline{\partial}\psi=\tfrac \varepsilon{-\varphi}\partial\overline{\partial}\varphi+\tfrac {\varepsilon}{\varphi^2}\partial\varphi\wedge\overline{\partial}\varphi.\]
\FormelQed
\end{proof}
\begin{lem}\label{lem:self-bounded}
Let $\alpha$ be a $(1,0)$-form and let $\omega$ be a positive $(1,1)$-form.
If $\ii \alpha\wedge\overline\alpha \leq \omega$, then $|\alpha|^2_\omega \leq \frac n2$.
\end{lem}
\begin{proof}
Let $e_j$ be an orthonormal base of $(1,0)$-forms with respect to $\omega$ (locally in $X_\reg$).
Then, $\omega=\frac \ii 2 \sum_j e_j\wedge \overline{e}_j$.
Writing $\alpha=\sum_j \alpha_j e_j$ gives $|\alpha|^2_\omega=\sum_j |\alpha_j|^2$. The assumption implies
\[0\leq \omega-\ii \alpha\wedge\overline\alpha=\ii\sum\nolimits_{j,k} \big(\tfrac12{\delta_{jk}}-\alpha_j\overline\alpha_k\big)e_j\wedge \overline{e}_k.\]
In particular, the trace of the Hermitian matrix associated to the form on the right side is non-negative,
\ie
\[0\leq \tr \big(\tfrac12\delta_{jk}-\alpha_j\overline\alpha_k\big)=\tfrac n2-\sum\nolimits_j |\alpha_j|^2.\]
\FormelQed[.7]
\end{proof}
For the convenience of the reader, let us also include a proof of the following basic fact.
\begin{lem}\label{lem:standard-self-bounded}
Let $z$ be the coordinates of $\CC^N$.
For $|z|<1$,  we get $\ii\partial|z|^2\wedge \overline{\partial}|z|^2 < \ii\partial\overline{\partial}|z|^2$.
\end{lem}
\begin{proof} By calculating the derivatives of $|z|^2$, we obtain
\[\partial\overline{\partial}|z|^2-\partial|z|^2\wedge \overline{\partial}|z|^2 = \sum\nolimits_{j,k} (\delta_{jk}-z_j\overline z_k) dz_j\wedge d\overline z_k.\]
The matrix $(\delta_{jk}-z_j\overline z_k)_{j,k}$ has only the eigenvalues $1$ (with multiplicity at least $N-1$) and
$1-|z|^2$ (with $z$ as an eigenvector). %
Hence, it is positive definite for $|z|^2<1$ and so,
the form $\ii\partial\overline{\partial}|z|^2-\ii\partial|z|^2\wedge \overline{\partial}|z|^2$.
\end{proof}
\smallskip \par
\begin{proof}[Proof of \autoref{main-local-thm}]
Let $X$ be an analytic set of pure-dimension $n$ in the unit ball $B:=\{z\in\CC^N\colon |z|<1\}$ and
let $D\Subset X$ be a relatively compact pseudoconvex domain in $X$ with $\sC^3$-smooth boundary.
By \autoref{lem:Ohsawa},
there exists a negative strictly plurisubharmonic function $\varphi$ on $D$
such that $\ii\partial\overline{\partial}\varphi$ induces a $\partial D$-complete Hermitian metric on $D\cap X_\reg$.
\par
Let $X'\subset X_\reg$ be the complement of an analytic set containing $X_\sing$ and
let $(E,h)$ be a holomorphic vector bundle on $X'$ with Nakano semi-positive metric $h$.
\par
Fix $u\in L^{2,\loc}_{n,q}(D,E)=L^{2,\loc}_{n,q}(D;E,h)$ which is $\overline{\partial}$-closed.
Then, there exists a smooth (strictly) plurisubharmonic positive exhaustion function $\Phi$ of $D$ %
\st $ue^{-\Phi}\in L^2_{n,q}(D;E,h)$.
Let $h'$ be the (twisted/weighted) metric on $E$ defined by $\|\cdot\|_{h'}=\|e^{-\Phi}\,\cdot\|_{h}$.
Then, $h'$ is also Nakano semi-positive and $u\in L^{2}_{n,q}(D;E,h')$.
We are going to show that there exists a $v\in L_{n,q-1}^2(D; E,h')$ with $\overline{\partial} v= u$.
This proves the vanishing of the $L^2$- and $L^{2,\loc}$-cohomology groups simultaneously,
\ie both equalities in \autoref{main-local-thm}.
\smallskip \par
There is a holomorphic function $f$ on the ball $B$ bounded by $1$ such that
\[X_\sing\subset X{\setminus} X' \subset A:=\{z\in B\colon f(z)=0\}.\]
Inspired by the work of W. L. Pardon and M. A. Stern (see \Sec 2 in \cite{PardonStern91}), we define the smooth plurisubharmonic functions
\[	\varphi_0(z):= |z|^2, \
\psi_1:= -\log (-\varphi), \
\psi_2:= -\log (-\log  |f|^2)\hbox{\ and\ }
\varphi_\varepsilon:= \varphi_0+\varepsilon\psi_1+\varepsilon\psi_2\]
on $D':=D{\setminus} A$.
Then, the K\"ahler forms $\omega_\varepsilon:=\ii\partial\overline{\partial} \varphi_\varepsilon$, $\varepsilon >0$, give complete metrics on $D'$,
which decrease pointwise monotonically to the metric $\omega_0:=\ii\partial\overline{\partial} \varphi_0$.
Thereby, the completeness with respect to $\partial D$ is ensured by the $\psi_1$ term in $\varphi_\varepsilon$
and the completeness with respect to $A$ is induced by the $\psi_2$ term.
Since $|z|<1$, \autoref{lem:standard-self-bounded} implies
$\ii\partial|z|^2\wedge \overline{\partial}|z|^2 \leq \omega_0\leq \omega_\varepsilon$.
Using \autoref{lem:logarithm-self-bounded} and $\varepsilon<1$, we obtain
$\ii\partial (\varepsilon\psi_k)\wedge\overline{\partial} (\varepsilon\psi_k)\leq\varepsilon\ii\partial\overline{\partial}(\varepsilon\psi_k)\leq\omega_\varepsilon$ for $k=1,2$, as well.
Hence, \autoref{lem:self-bounded} implies
\[\big|\partial \varphi_\varepsilon\big|^2_{\omega_\varepsilon}
\leq \big|\partial |z|^2\big|^2_{\omega_\varepsilon}
+\big|\partial (\varepsilon\psi_1)\big|^2_{\omega_\varepsilon}
+\big|\partial (\varepsilon\psi_2)\big|^2_{\omega_\varepsilon}
\leq \tfrac {3}2 n,\]
\cf \cite[\Lem 2.4]{PardonStern91} (while our arguments simplify their proof).
Hence, we can apply \autoref{thm:DFO} (DFO)
with respect to every metric $\omega_\varepsilon$,
\ie for each $\varepsilon>0$,
there exists a $v_\varepsilon\in L^2_{n,q-1}(D',\omega_\varepsilon;E,h')$
with $\overline{\partial} v_\varepsilon=u$ and $\|v_\varepsilon\|_{\omega_\varepsilon,h'}\leq 4\cdot \frac32 n\|u\|_{\omega_\varepsilon,h'}$.
\par
Since the estimates are independent of $\varepsilon$,
we can apply the approximation theorem from below, \autoref{thm:incomplete-approximated-by-complete}, and
get a solution $v\in L^2_{n,q-1}(D',\omega_0; E, h')$ %
of $\overline{\partial} v= u$ with $\|e^{-\Phi}v\|_{\omega_0,h}=\|v\|_{\omega_0,h'}\leq 6 n\|u\|_{\omega_0,h'}$.
\par
A priori, the equation $\overline{\partial} v= u$ is only satisfied on $D'=D{\setminus} A$.
Yet, since we consider the Dolbeault operator in sense of distribution on square-integrable forms,
the equation extends beyond analytic sets by an $L^2$-version of Riemann's extension theorem,
see \eg \cite[\Thm 3.2]{Ruppenthal09IntJMath} or \cite[\Thm 5.6]{Sera15}.
That means $\overline{\partial} v=u$ on $D\cap X'$. %
\end{proof}
\smallskip \par
We used that if an incomplete Hermitian metric $\omega$ can be approximated in an appropriate way by complete metrics $\omega_k$,
then the vanishing of $L^2$-Dolbeault $(n,q)$-cohomology groups with respect to $\omega_k$
implies also the vanishing with respect to $\omega$.
The precise statement of this is as follows %
(see also \cite[\Thm 4.1]{Demailly82} and \cite[\Prop 4.1]{Ohsawa87Hodge}).
\begin{thm}[\Lem 2.3 in \cite{PardonStern91}]
\label{thm:incomplete-approximated-by-complete}
Let $M$ be a complex manifold,
let $E\rightarrow M$ be a holomorphic vector bundle on $M$ and %
let $\{\omega_k\}$ be a pointwise decreasing sequence of Hermitian metrics of $M$
which converges pointwise to a Hermitian metric $\omega_0$ of $M$.
If the $\overline{\partial}$\nbd{}equation is solvable for $\overline{\partial}$\nbd{}closed forms
in $L^2_{n,q}(M, \omega_k; E)$ with an estimate independent of $k$,
then the $\overline{\partial}$\nbd{}equation is solvable for $\overline{\partial}$\nbd{}closed forms
in $L^2_{n,q}(M, \omega_0; E)$ with the same estimate.
\end{thm}
The proof in \cite{PardonStern91}  straight forwardly generalizes to
forms with values in vector bundles without any crucial changes.
\medskip \par
\section{Equivalence of \texorpdfstring{$L^2$- and $L^{2,\loc}$-}{L{\000\136}2- and L\000\136\{2,loc\}-}cohomology groups}
\label{sec:equivalence}
\smallskip \par
Throughout this section, we are in the following setting (as assumed in \autoref{main-grauert}). %
Let $X$ be a Hermitian complex space of pure-dimension $n$,
let $X'\subset X_\reg$ be the complement of an analytic set containing $X_\sing$ and
let $E\rightarrow X'$ be a holomorphic vector bundle with a metric $h$ which is locally Nakano semi-positive with respect to $X$,
see \autoref{df:loc-Nakano}.
\par
Let $\Omega=\{\rho<0\}\Subset X$ be a relatively compact  domain in $X$
where $\rho\colon X\rightarrow\RR$ is a $\sC^3$-smooth function
such that $d\rho\neq 0$ close to $\partial\Omega$ and
$\rho$ is strictly plurisubharmonic on a neighbourhood $U$ of $\partial \Omega$ except of finitely many points $x_1{,}...,x_m$. %
We can assume that $x_1{,}...,x_m$ is contained in $\partial\Omega$. %
We set $\Omega_\varepsilon:=\{\rho<\varepsilon\}$ for any number $\varepsilon$.
Let us assume that $U=\Omega_{\varepsilon_0}{\setminus}\overline{\Omega_{-\varepsilon_0}}$ for some fixed $\varepsilon_0>0$.
\par
Let $\iota\colon H^{n,q}_{w}(\Omega,E) \rightarrow H^{n,q}_{w,\loc}(\Omega,E)$ be the natural morphism induced by the embedding $L^2_{n,q}(\Omega,E) \hookrightarrow L^{2,\loc}_{n,q}(\Omega,E)$.
For $q>0$,
the goal of this section is to prove that $\iota$ is an isomorphism (proving \autoref{main-grauert}),
giving us the equivalence of $H^{n,q}_{w}(\Omega,E)$ and $H^{n,q}_{w,\loc}(\Omega,E)$.
\medskip \par
\subsection{Surjectivity}\label{ssec:surjectivity}
In order to prove that $\iota$ is surjective,
we extend (locally) square-integrable forms (with values in $E$) across the boundary of $\Omega$.
Of course, we cannot extend every (locally) square-integrable form.
Yet with Grauert's bump method (developed in \cite{Grauert58}),
we get an extension of $f$'s cohomology class.
This statement is already well-known in the literature for domains in manifolds whose boundary is strictly pseudoconvex everywhere, \cf \eg \S\,IV.7 in \cite{LiebMichel02}.
As the author is not aware of any published proof for the singular setting,
we present a proof which is similar to the one in the smooth case, yet includes the necessary adaptions to accommodate for the weaker assumption on the boundary
in our setting above (\ie as assumed in \autoref{main-grauert}).
\smallskip \par
Let $U_i\Subset U, i=1{,}...,t$, be open and pseudoconvex \st $\partial\Omega\subset\bigcup_i U_i$ and $U_i$ can be embedded in the unit ball of a $\CC^{N_i}$. %
We also assume that each point of $x_1{,}...,x_m$ is not contained in more than one $U_i$
and there are Nakano semi-positive metrics $h_i$ of $E|_{U_i}$
which are equivalent to $h$ on $U_i$, see \autoref{df:loc-Nakano}.
We choose $\psi_i\in \sC^\infty(X,[0,1])$ \st $\supp \psi_i\subset U_i$, $\sum \psi_i>0$ on $b\Omega$,
$\psi_i$ is constant on $\{x_1{,}...,x_m\}\cap U_i$  and
\[\rho_i:=\rho-\sum\nolimits_{k=1}^i \psi_i\in\sC^3(X)\cap\SPSH(U'),\]
where $\SPSH(U')$ is %
the space of strictly plurisubharmonic functions on $U':=U{\setminus}\{x_1{,}..,x_m\}$.
We set $\Omega_i:=\{\rho_i<0\}$, see \autoref{fig:bump}, and may assume that $d\rho_i\neq0$ close to $\partial\Omega_i$.
By construction, $\Omega_0=\Omega$, $\Omega_i\Subset \Omega_{\varepsilon_0}$ and $\Omega_i$ has a $\sC^3$-smooth pseudoconvex boundary which is even strictly pseudoconvex except of in $\{x_1{,}...,x_m\}\cap\partial\Omega_i$.
\par
\setcounter{figure}{\arabic{equation}}\stepcounter{equation}
\begin{figure}[!b]
\includegraphics{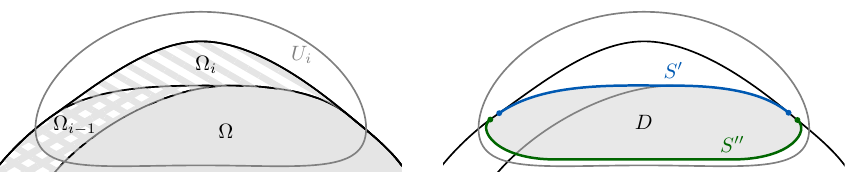} %
\caption{Sketches of the bump.}%
\label{fig:bump}
\end{figure}
\begin{lem}\label{Aa:lem}
(i) For every $\overline{\partial}_w$-closed $f\in L^2_{n,q}(\Omega_{i-1},E)$, there exist a $\overline{\partial}_w$-closed $f'\in L^2_{n,q}(\Omega_i,E)$ and a $g\in L^{2}_{n,q-1}(\Omega_{i-1},E)$ in $\dom\overline{\partial}_w$ \st on $\Omega_{i-1}$,
\[f=f'+\overline{\partial} g.\]
\par
(ii) The analogue $L^{2,\loc}$-statement holds as well.
\end{lem}
\begin{proof}
Let $\Phi$ denote a plurisubharmonic exhaustion function of $U_i$ such that $\Omega_i{\setminus}\Omega_{i-1} \Subset\{\Phi < 0\}$.
By taking the composition with a smooth convex increasing function $\RR \rightarrow\RR$ which is $0$ on $\RR^{-}$, %
we can assume that $\Phi$ is 0 on $\Omega_i{\setminus}\Omega_{i-1}$.
\par
Then,
$D:=\{\rho_{i-1}+\Phi < 0\} \subset \Omega_{i-1}$
is a $\sC^3$-smooth pseudoconvex bounded domain
which satisfies the following two properties: $S':= \overline{\partial \Omega_{i-1} \cap \Omega_i}$ is contained in $\partial D$ and $S'':= \overline{\partial D\cap \Omega_{i-1}}$ is disjoint from $S'$ (see \autoref{fig:bump}).
Let $f\in L^2_{n,q}(\Omega_{i-1},E)$ be $\overline{\partial}$-closed.
By \autoref{main-local-thm},
there is a solution $\tilde{g}\in L^{2}_{n,q-1}(D,E)$ for the equation $f=\overline{\partial} \tilde{g}$ on $D$.
We obtain this since our metric on $X$ embedded in the complex number space is equivalent to the standard metric on $D$ (induced by an embedding in the complex number space).
\par
Let $\chi\in \sC^{\infty}(X,[0,1])$ be a cut-off function
with $\supp \chi \subset U_i$, $\chi=0$ on a neighbourhood of $S''$  and $\chi|_{\Omega_i{\setminus}\Omega_{i-1}}=1$.
This $\chi$ can be chosen since $S'\cap S''=\emptyset$. We set
\[g:=\left\{ \begin{array}{ll} \chi\cdot \tilde g & \hbox{on } D,\\  0 & \hbox{on }\Omega_{i-1}{\setminus} D.\end{array}\right.\]
Then, $g$ is in $L^2_{n,q}(\Omega_{i-1},E)\cap \dom\overline{\partial}_w$. Now, we can set
\[f':=\left\{ \begin{array}{ll} f-\overline{\partial} g & \hbox{on }  \Omega_{i-1},\\ 0 & \hbox{on } \Omega_i{\setminus}\Omega_{i-1}.\end{array}\right.\]
$f'$ is well-defined since
\[\overline{\partial} g=\overline{\partial}\left(\chi\cdot \tilde g \right)= \overline{\partial} \tilde g = f\]
on a neighbourhood of $S'=\overline{\partial \Omega_{i-1} \cap \Omega_i}$ in $\Omega_{i-1}$. By construction, $f'\in L^2_{n,q}(\Omega_i,E)\cap\ker \overline{\partial}_w$.
\smallskip \par
(ii) can be proven by exactly same argumentation as (i).
\end{proof}
\begin{rem}\label{rem:bump}
Just considering one single step $i$,
it was not used that
$\rho_i$ is strictly plurisubharmonic outside of $x_1{,}...,x_m$ %
to obtain the extension of $[f]$ from $\Omega_{i-1}$ to $\Omega_{i}$.
We just used that $D$ is pseudoconvex (as more is not needed in \autoref{main-local-thm}),
\ie $\Omega_{i-1}$'s boundary is pseudoconvex and
$S'= \overline{\partial \Omega_{i-1} \cap \Omega_i}$ is disjoint from $S''= \overline{\partial D\cap \Omega_{i-1}}$.
In other words, $\Omega_i$ can be an arbitrary bump changing just a small part of the boundary
although it does not necessarily have a pseudoconvex boundary itself.
\end{rem}
\smallskip \par
As $\Omega_\varepsilon\subset \Omega_t$ for small enough $\varepsilon$, applying \autoref{Aa:lem} recursively gives us the following.
\begin{cor}\label{surjectivity}
(i) For $q>0$ and for every $f\in L^2_{n,q}(\Omega,E)$ with $\overline{\partial} f=0$,
there exist an $f'\in L^2_{n,q}(\Omega_\varepsilon,E)\cap\ker\overline{\partial}_w$ and
a $g\in L^{2}_{n,q-1}(\Omega,E)\cap \dom\overline{\partial}_w$ with $f=f'+\overline{\partial} g$ on $\Omega$.
In particular, the map $R\colon H_{w,\loc}^{n,q}(\Omega_\varepsilon,E)\rightarrow H_w^{n,q}(\Omega,E)$
induced by the restriction is surjective for all $q>0$.
\par
(ii) The analogue $L^{2,\loc}$-statement holds as well, \ie the map $R_\loc\colon H_{w,\loc}^{n,q}(\Omega_\varepsilon,E)\rightarrow H_{w,\loc}^{n,q}(\Omega,E)$ induced by the restriction is surjective for all $q>0$.
\end{cor}
\begin{rem}\label{rmk:surjectivity-in-weaker-setting}
Actually, the statement of \autoref{surjectivity} is correct as well
if $\Omega$'s boundary is just pseudoconvex in $x_1{,}...,x_m$
and strictly pseudoconvex in all other points,
\ie there is (only) a defining function $\rho$ of $\Omega$
which is strictly plurisubharmonic on a neighbourhood of $\partial\Omega{\setminus}\{x_1{,}...,x_m\}$
and not necessarily strictly plurisubharmonic on $U{\setminus}\{x_1{,}...,x_m\}$.
To prove this,
we can first extend $[f]$ along the boundary $\partial\Omega$ without arbitrary small neighbourhoods of $x_1{,}...,x_m$;
then, we can extend $[f]$ in $x_1{,}...,x_m$ following \autoref{rem:bump}.
\end{rem}
\medskip \par
\subsection{Injectivity}%
For the proof of the injectivity, we are using the Remmert reduction. Therefore, let us recall the following first.
\begin{lem}[Satz 3 in \S\,2 of \cite{Grauert62}]\label{Abeta:lem}
If $D\Subset X$ is a domain with a strictly pseudoconvex, smooth boundary,
then there exists a compact set $K\subset D$
such that each compact analytic set $A\subset D$, $\dim A>0$, is contained in $K$.
\end{lem}
Although, some points of $\Omega$'s boundary might not be strictly pseudoconvex in our setting,
$\Omega_{-\varepsilon}=\{\rho<-\varepsilon\}$'s boundary is strictly pseudoconvex everywhere for all small enough $\varepsilon>0$.
Following Remmert, the degeneracy set, which is the union of all compact analytic sets of positive dimension inside a complex space, is an analytic set, see \cite{Remmert57}.
With \autoref{Abeta:lem} for $D=\Omega_{-\varepsilon}$,
we get a maximal compact analytic set $A_{\max}$ in $\Omega_{-\varepsilon}$.
So, $A_{\max}$ also maximal in $\Omega$ and $\Omega_\varepsilon$ in our setting.
\smallskip \par
\begin{thm}\label{R-loc-injectivity}
In the setting above, we assume also that $\KE$%
\footnote{Please recall that $\KE$ is the sheaf of locally square-integrable weakly holomorphic $n$\nbd forms with values in $E$,
\ie $\KE(B)$ is the kernel of $\overline{\partial}_{w,\loc}$ in $L^{2,\loc}_{n,0}(B,E)$ for any open $B\subset X$.}
is a coherent analytic sheaf on $X$.
Then,
the map $R_\loc\colon H^{n,q}_{w,\loc}(\Omega_\varepsilon,E)\rightarrow H^{n,q}_{w,\loc}(\Omega,E)$
induced by the restriction is injective for all small enough $\varepsilon>0$ and $q>0$.
\end{thm}
For $X$ being a manifold and $\partial\Omega$ being everywhere strictly pseudoconvex,
this fact is well-known, see for instance \Thm 3.14 in \cite[\Chap VI]{LiebMichel02}.
Although the proof is not much different in our setting, we include it for the convenience of the reader.
\begin{proof}
Let $\varepsilon>0$ be so small such that $\partial\Omega_\varepsilon$ is strict pseudoconvex,
\ip $\Omega_\varepsilon$ is holomorphically convex. Hence, we have the Remmert reduction $\pi\colon \Omega_\varepsilon\rightarrow Z$, where $Z$ is a Stein space, $\pi$ is proper and holomorphic, $\pi(A)=Z_0$ and $\pi\colon\Omega_\varepsilon{\setminus} A\rightarrow Z{\setminus} Z_0$ is biholomorphic, see \cite{Remmert56}.
$A=A_{\max}$ is the maximal compact analytic set in $\Omega\subset\Omega_\varepsilon$ and
$Z_0$ is the union of finitely many points
since it is compact and analytic in a Stein space.
\smallskip \par
We choose an open Stein cover $\aU_\varepsilon=\{U_i\}_{i\in I}$ of $\Omega_\varepsilon$
with $I=I_1 {\dot \cup}  I_2 {\dot \cup}  I_3$ such that
\AddItem[Omega-cover]~$\aU:=\{U_i\}_{i\in I_1\cup I_2}$ covers $\Omega$,
\AddItem[disjointX] $U_{ij}=\emptyset$\,\footnote{%
We use the usual index notation for intersections, \ie
$U_{i_0...i_q}:=\bigcap_{k=0}^qU_{i_k}$.%
}
for $i\in I_1$ and $j\in I_3$ and
\AddItem[unchangedX] $U_i\cap A=\emptyset$ for $i\in I_2\cup I_3$.
\par
By \autoref{main-local-thm} (\ip using that $E$ is locally Nakano semi-positive with respect to $X$),
$\overline{\partial}_{w,\loc}$ gives us a free resolution of $\KE$, see also \cite[\Thm 3.1]{Ruppenthal14Duke} and \cite[\Thm 10.5]{Sera15}.
Therefore, the Formal de Rham Lemma (see \eg \S\,1.3 in \cite[\Chap B]{GrauertRemmertTSS})
infers that the $\overline{\partial}_{w,\loc}$-Dolbeault $(n,q)$-cohomology group is isomorphic to the $q$-sheaf cohomology group of $\KE$.
We get the following commutative diagram:
\[\xymatrix{
H^{n,q}_{w,\loc}(\Omega_\varepsilon,E)\ar[d]_{R_\loc}\ar[r]^{\sim} & H^q(\Omega_\varepsilon,\KE[\Omega_\varepsilon])\ar[d]\ar[r]^{\sim}& \check H^q(\aU_\varepsilon,\KE[\Omega_\varepsilon])\ar[d]\\
H^{n,q}_{w,\loc}(\Omega,E) \ar[r]^{\sim} & H^q(\Omega,\KE[\Omega])\ar[r]^{\sim}& \check H^q(\aU,\KE[\Omega])\DiagrammPunkt}\]
Hence, the injectivity of $R_\loc$ is proven if we can show
that each cocycle $\eta\in Z^q(\aU_\varepsilon,\KE[\Omega_\varepsilon])$
with $\eta|_{\aU}:=(\eta_{i_0...i_q})_{i_k\in I_1\cup I_2}=\delta\xi'$ for a  $\xi'\in C^{q-1}(\aU,\KE[\Omega])$
is also a coboundary on $\aU_\varepsilon$,
\ie $\eta=\delta\xi$ for a $\xi\in C^{q-1}(\aU_\varepsilon,\KE[\Omega_\varepsilon])$.
\smallskip \par
We define $\xi^{(1)}=(\xi^{(1)}_{i_1...i_q})\in C^{q-1}(\aU_\varepsilon,\KE[\Omega_\varepsilon])$ as the trivial extension of $\xi'$, \ie
\[\xi^{(1)}_{i_1...i_q}:=
\bigg\{\begin{array}{cl} 0& \hbox{ if  } \exists k \hbox{ \st } i_k\in I_3,\\
\xi'_{i_1...i_q} & \hbox{ otherwise,} \end{array}
\]
$\eta^{(1)}:=\delta\xi^{(1)}\in B^q(\aU_\varepsilon,\KE[\Omega_\varepsilon])$ and $\eta^{(2)}:=\eta-\eta^{(1)}\in Z^q(\aU_\varepsilon,\KE[\Omega_\varepsilon])$.
By construction, %
$\eta^{(2)}_{i_0...i_q}=0$ if all $i_0{,}...,i_q\in I_1\cup I_2$. %
Since $U_{ij}=\emptyset$ for $i\in I_1$ and $j\in I_3$ by \GetItem{disjointX}, we get  %
\begin{equation}\label{eq:easy-eta}
\eta^{(2)}_{i_0...i_q}=0\hbox{ or }U_{i_0...i_q}=\emptyset
\quad\hbox{if  } \exists k \hbox{ \st } i_k\in I_1.
\end{equation}
\smallskip \par
Let $\aV:=\{V_i\}_{i\in \tilde I}$ be an open Stein cover of $Z$
with $\tilde I=\tilde I_1 {\dot \cup}  I_2 {\dot \cup}  I_3$ such that
\AddItem[unchangedY]
$V_i=\pi(U_i)$ for all $i\in I_2\cup I_3$,
\AddItem[disjointZ]
$V_{ij}=\emptyset$ for $i\in \tilde I_1$ and $j\in I_3$ and
\AddItem[coarse]
for every $i\in I_1$, there is a $j\in \tilde I_1$ such that $\pi(U_i)\subset V_j$.
If necessary, we can make $\aU_\varepsilon$ finer to obtain \GetItem{coarse}.
By \GetItem{unchangedX} and since $\pi$ is one-to-one outside of $A$,
we get for every $i\in I_2\cup I_3$, $V_i$ is indeed open, Stein and disjoint from $Z_0=\pi(A)$.
$\tilde{\aU}:=\{\tilde U_i\}_{i\in \tilde I}$ defined by $\tilde U_i:=\pi^{-1}(V_i)$ %
is an open cover of $\Omega_\varepsilon$ (not Stein necessarily).
Actually, $\tilde U_i=U_i$ for all $i\in I_2\cup I_3$. %
Because of this and \GetItem{coarse}, $\aU_\varepsilon$ is finer than $\tilde{\aU}$.
Let $\tau\colon I\rightarrow\tilde I$ denote the refinement map.
Using \eqref{eq:easy-eta}, we can define well
\[
\tilde\eta_{i_0...i_q}
:=\bigg\{\begin{array}{cl} 0& \hbox{ if  } \exists k \hbox{ \st } i_k\in\tilde I_1,\\
\eta^{(2)}_{i_0...i_q} & \hbox{ otherwise.} \end{array}\]
Straight forward calculations confirm $\tilde\eta$ is a cocycle, \ie $\tilde\eta\in Z^q(\tilde{\aU},\KE[\Omega_\varepsilon])$.
Actually, $\eta^{(2)}$ is the image of  $\tilde\eta$ under the map $C^q(\tau)$ induced by the refinement $\tau$ on cochains/cocycles.
Since $\tilde\eta_{i_0...i_q}=0$ whenever $U_{i_0...i_q}\cap A\neq\emptyset$, we get
$\vartheta_{i_0...i_q}=\tilde\eta_{i_0...i_q}\circ\pi^{-1}$ defines a cocycle $\vartheta\in Z^q(\aV,\pi_\ast\KE[\Omega_\varepsilon])$ on $Z$.
\par
Since $Z$ is Stein and $\pi_\ast\sK_{E,\Omega_\varepsilon}$ is a coherent analytic sheaf,
Theorem B of Cartan--Serre gives us $\check H^q(\aV,\pi_\ast\KE[\Omega_\varepsilon])=0$, %
\ie there exists a $\zeta\in C^{q-1}(\aV,\pi_\ast\sK_{E,\Omega_\varepsilon})$
with $\vartheta=\delta\zeta$. %
We set $\tilde\xi:=\pi^\ast \zeta \in C^{q-1}(\tilde{\aU},\KE[\Omega_\varepsilon])$ and get $\tilde\eta=\delta\tilde\xi$.
As $\aU_\varepsilon$ is finer than $\tilde{\aU}$, the refinement map $\tau$ induces
a cochain $\xi^{(2)}=C^{q-1}(\tau)(\tilde\xi)\in C^{q-1}(\aU_\varepsilon,\KE[\Omega_\varepsilon])$ such that $\eta^{(2)}=\delta\xi^{(2)}$.
Hence, %
\[\eta=\eta^{(1)}+\eta^{(2)}=\delta\left(\xi^{(1)}+\xi^{(2)}\right).\]
\FormelQed
\end{proof}
\smallskip \par
\begin{proof}[Proof of \autoref{main-grauert}]
We consider the following commutative diagram:
\[\xymatrix{
& H^{n,q}_{w,\loc}(\Omega_\varepsilon,E)\ar[d]^{R_\loc}\ar[dl]_{R} \\
H^{n,q}_{w}(\Omega,E) \ar[r]_{\iota} &H^{n,q}_{w,\loc}(\Omega,E)\DiagrammPunkt}\]
Since $R_\loc$ is surjective by \autoref{surjectivity} (ii),
we get $\iota$ is surjective.
By \autoref{surjectivity} (i), we get also $R$ is surjective,
\ip the kernel of $\iota$ is equal to $R$'s image of the kernel of $\iota \circ R=R_\loc$.
Since \autoref{R-loc-injectivity} implies the latter is 0, %
$\iota$ is injective, as well.
\end{proof}
\bibliographystyle{amsalpha}
\providecommand{\bysame}{\leavevmode\hbox to3em{\hrulefill}\thinspace}
\providecommand{\MR}{\relax\ifhmode\unskip\space\fi MR }
\providecommand{\MRhref}[2]{%
\href{http://www.ams.org/mathscinet-getitem?mr=#1}{#2}
}
\providecommand{\href}[2]{#2}

\end{document}